\documentclass[12pt]{amsart}
 \usepackage{latexsym,amsthm}
\usepackage{amssymb}
\usepackage{amsmath}
\usepackage{epsfig}
\usepackage{graphicx}
 \usepackage{mathrsfs}
\numberwithin{equation}{section}
\usepackage{cite}
\usepackage{comment}
\usepackage{subfiles}
\usepackage{cleveref}
\input xy
\xyoption{all}

\usepackage[usenames,dvipsnames]{color}

\newenvironment{psmallmatrix}
  {\left(\begin{smallmatrix}}
  {\end{smallmatrix}\right)}
\theoremstyle{plain}
\newtheorem{prop}{Proposition}[section]

\newtheorem{theo}[prop]{Theorem}
\newtheorem{coro}[prop]{Corollary}

\newtheorem{lemm}[prop]{Lemma}

\theoremstyle{definition}
\newtheorem{defi}[prop]{Definition}
\newtheorem{prob}[prop]{Problem}

\newtheorem{conj}[prop]{Conjecture}

\newtheorem{exam}[prop]{Example}

\numberwithin{equation}{section}

\newcommand{\GL}{{\rm GL}}
\newcommand{\PSL}{{\rm PSL}}

\newcommand{\SL}{{\rm SL}}

\newcommand{\F}{\mathbb{F}\,}
\newcommand{\core}{{\rm core}}
\newcommand{\glq}{{\GL_n(\F_q)}}
\newcommand{\slq}{{\SL_n(\F_q)}}
\newcommand{\pslq}{{\PSL_n(\F_q)}}
\newcommand{\lcm}{{\rm lcm}}

\begin{document}
\title[Minimal Faithful Permutation Representations]
{Minimal Faithful Permutation Representations for linear groups}

\author{Neelima Borade}
\author{Ramin Takloo-Bighash}
\address{Department of Mathematics, Statistics, and Computer Science, \\
University of Illinois at Chicago, 851 S. Morgan Str\\
Chicago, IL 60607\\
USA}
\email{nb4296@princeton.edu}
\email{rtakloo@uic.edu}

\begin{abstract} 
In this paper we study the minimal faithful permutation representations of $\SL_n(\F_q)$ and $\GL_n(\F_q)$. 
\end{abstract} 


\date{\today}
\maketitle

\section{Introduction}
By a classical theorem of Cayley any finite group can be realized as a subgroup of a finite permutation group. In fact, for a finite group $G$ of size $|G|$, Cayley's construction identifies $G$ with a subgroup of $S_{|G|}$, and the embedding is given by the regular action of $G$ on itself. One can often find a smaller permutation group $S_n$ containing a copy of $G$.  Given a finite group $G$ we define $p(G)$ to be the smallest natural number $n$ such that $S_n$ has a subgroup isomorphic to $G$.  We call the corresponding permutation representation on $n$ letters the {\em minimal faithful representation of $G$}. By the proof of Cayley's theorem, $p(G) \leq |G|$. Despite its innocent sounding definition, computing $p(G)$ in general is an unsolved problem. Johnson's paper \cite{Johnson} seems to be the first reference which addresses the problem of finding $p(G)$ as such and obtains various results, among which is the classification of finite groups $G$ such that $p(G) = |G|$, \cite[Theorem 1]{Johnson}.  Johnson \cite{Johnson} also gives the value of $p(G)$ for Abelian groups. Various other classes of groups, including $p$-groups, some easy semi-direct products, and some solvable groups, are studied in \cite{Elias, EST}, though no explicit general formulae are obtained. 

The study of $p(G)$ for classical simple groups goes back to Galois who proved if $q >11$ is a prime number, then $p({\rm{PSL}}_2(\F_q)) = q+1$. 
Summarizing results by Cooperstein, Grechkoseeva, Mazurov, and Vasil'ev   \cite{Cooper, Gre, Maz, V1, V2, V3, V4},  Table 4 of \cite{Guest} contains values of $p(G)$ for classical simple groups and exceptional simple groups of Lie type.  The Atlas  \cite{Conway} contains the value of $p(G)$ for all finite sporadic simple groups. For a simple group $G$ computing $p(G)$ and determining the corresponding minimal faithful representation is tantamount to finding a subgroup of $G$ of minimal index. 
In his thesis, Patton \cite{Patton} determined  subgroups of minimal index in $\SL_n(\F_q)$ and  $\rm{SP}_{2m}(\F_q)$ for $q$ an odd prime power. 
Cooperstein \cite{Cooper} computed the minimal index of a subgroup for the remaining classical groups over finite fields using a generalization of Patton's method.   Given $G$, if $H$ is the subgroup of minimal index in $G$, we obtain a permutation representation of $G$ on $ G/H$. Cooperstein calls the size of the set $\left| G/H \right|$ the {\em minimal degree of $G$}. In this paper, we denote the minimal degree of a group $G$ by $\rho(G)$.
An important point to note is that this permutation representation is, in general, {\em not} faithful, and for that reason, unless $G$ is simple, the minimal degree $\rho(G)$ of $G$ is not necessarily equal to $p(G)$.\footnote{Even some of the experts reading an early draft of this manuscript, including at least one referee, have been confused about the distinction between the minimal degree $\rho(G)$ a la Cooperstein and $p(G)$.}  In fact, Table 1 of \cite{Cooper} has a column listing the size of the kernel of the corresponding permutation representation for each case. 

In general, computing $p(G)$ for a non-simple group $G$, and determining the structure of the corresponding faithful permutation representation, is far more complicated than computing the minimal degree $\rho(G)$ of $G$. For comparison, if $G$ is a group of order $q^n$ for a prime number $q$, then by Sylow's theorem, $G$ has a subgroup of order $q^{n-1}$, and as a result, $\rho(G)$ is always equal to $q^n/ q^{n-1} = q$. On the other hand, by \cite{Elias, EST}, the quantity $p(G)$ is related to the fine structure theory of $G$.

There are very few results in literature about the computation of $p(G)$ for a non-simple finite group $G$ of Lie type.  In fact, to the best of our knowledge, save for the few examples worked out in the Atlas \cite{Conway}, the only known family of example is $G = \SL_2(\F_q)$ for $q$ a prime power. Theorem 3.7 of \cite{gl2} claims to have computed $p(\GL_2(\F_q))$ but as we will see, this result is wrong; the correct result for this case is stated as Theorem \ref{theo11} below where we bootstrap the result on $\SL_2$ to build minimal faithful representations of $\GL_2$.    The passage from $\SL_2$ to $\GL_2$ is not trivial. Observe that $\GL_2(\F_q)= D_1 \cdot \SL_2(\F_q)$, where $D_1$ is the group of diagonal matrices with any $a \in {\F_q}^\times$ as the top left element on its diagonal and every other element equal to $1$. This might suggest that the degree of the minimal faithful permutation representation for $\GL_2(\F_q)$ can be easily computed from that of $\SL_2(\F_q)$ by utilizing the semi-direct product decomposition. In fact, Lemma 2.4 in \cite{EDH} claims that the degree of the minimal faithful permutation representation of the semi-direct product $G \rtimes H$ is the same as that of $G$. However, there is a mistake in  line six of  the proof. The authors' claim that if $B_1, \cdots, B_k$ gives a minimal faithful representation of $G$, then the core in $GH$ of $B_1H \cap \cdots \cap B_kH = $ core of $(B_1 \cap \cdots \cap B_k)H$  is  the core in $GH$ of $B_1 \cap \cdots \cap B_k$ times the core in $GH$ of $H$.  This is not necessarily true, and in fact false in our case.  In general the subgroup structure of semi-direct product is very complicated, and by \cite[Lemma 1.3]{prodsubgp}, maximal subgroups of the direct product of groups depend on both groups and are typically quite intricate. Even for a familiar group like the dihedral group the minimal faithful permutation representation is surprisingly small, \cite{Elias}. 

\

This is a summary of what we accomplish in this paper: 

\begin{enumerate}
\item[(i)] We determine $p(\SL_3(\F_q))$ for all odd prime powers $q \geq 3$, c.f. Theorem \ref{sl3}. 
\item[(ii)] When $n \geq 4$, we compute $p(\SL_n(\F_q))$ for infinitely many pairs $(n, q)$, including the cases where $\gcd(n, q-1)=1$, c.f. Theorem \ref{slnmainres}. 
\item[(iii)] We formulate a conjecture for $p(\SL_n(\F_q))$ for all large enough odd prime powers $q$ and all $n$, c.f. Conjecture \ref{conj:sln}. 
\item[(iv)] We classify all minimal faithful permutation representations of $\GL_2(\F_q)$ for $q$ an odd prime power $\geq 3$ and compute $p(\GL_2(\F_q))$, c.f.  Theorem  \ref{theo11}. 
\item[(v)] We classify all minimal faithful permutation representations of $\GL_3(\F_q)$ for $q$ a prime power $\geq 3$ and compute $p(\GL_3(\F_q))$, c.f. Theorem \ref{theo314}.  
\item[(vi)] For all $n, q$, we construct a faithful permutation representation of $\GL_n(\F_q)$ and we conjecture that it is minimal, c.f. Proposition \ref{fcoll} and Conjecture \ref{conj11}. Also, see Corollary \ref{coro25}. 
\end{enumerate}

As in the case of $\SL_2(\F_q)$ treated in  \cite{sl2}, our computation of $p(\SL_3(\F_q))$ uses the classification of the maximal subgroups, Tables $8.3$ and $8.4$ of \cite{sln}.   The fact that our proof of this fact relies on the classification of subgroups is an impediment to computing $p(\SL_n(\F_q))$ in general. For $n \geq 4$, starting with Patton's subgroup \cite{Patton} we construct in the proof of Theorem \ref{slnmainres} a core-free collection of subgroups of $\SL_n(\F_q)$. In general we cannot prove that this collection gives a minimal faithful permutation representation, but we can do so if $n, q$ satisfy a technical condition, {\em very divisibility} in Definition \ref{def:div}.  In cases where there is a classification of maximal subgroups, e.g. as in the tables in \cite{sln},  we have verified that these  core-free collections are in fact minimal. An illustrative example is explained in Example \ref{ex:special}. Inspired by these examples we conjecture that the core-free collections constructed in the proof of Theorem \ref{slnmainres} are in general minimal, Conjecture \ref{conj:sln}.

To construct faithful collections of $\GL_n(\F_q)$ we use our knowledge of faithful collections of $\SL_n(\F_q)$.  By using the faithful collection constructed for $\SL_n(\F_q)$ we construct a faithful collection for $\GL_n(\F_q)$ in \S \ref{sect:amincollection}.  In \S \ref{sec:mingl2} and \S\ref{sec:mingl3}, we prove that this faithful collection is in fact minimal for $\GL_2(\F_q)$ and $\GL_3(\F_q)$, respectively.  Here too the proof relies on the classification of maximal subgroups of $\SL_2(\F_q)$ and $\SL_3(\F_q)$. The idea is to start with an arbitrary minimal faithful collection of $\GL_2(\F_q)$, respectively $\GL_3(\F_q)$, and try to beat the bound given by the faithful collection of  \S \ref{sect:amincollection}.  The surprising thing is that for $\GL_2(\F_q)$ and $\GL_3(\F_q)$ the minimal faithful collections are very rigid and one can completely classify them, Theorem \ref{theo313} and Theorem \ref{theo314}. What facilitates translating the minimality problem from $\GL_n(\F_q)$ to $\SL_n(\F_q)$ is a clever idea from \cite{gl2} which we formulate as Lemma \ref{lemm31} and Lemma \ref{lemm32}.  In this case too we can eliminate our reliance on the subgroup classification if we assume $n, q$ satisfy the technical condition of Definition \ref{def:div}.   These considerations, plus other numerical examples, motivate Conjecture \ref{conj11}. 
  
It is clear from our presentation that much of what we do here can be extended to other finite classical groups and their similitude counterparts.  It would be interesting to investigate quasi-permutation representations of linear groups \cite{Wong}. Behravesh and his collaborators \cite{sl2, gl2} have started this investigation, but there is much to be done. Understanding minimal faithful permutation representations and their cousins, minimal faithful quasi-permutation representations, is not only of clear inherent interest, but also of importance in computations. For example, \cite{Applications} highlights some applications of minimal faithful permutation representations in computational group theory.  For another class of applications, see \cite{GSL}.

\

This paper is organized as follows.  Section \ref{sec:sln} contains the results on $\SL_n(\F_q)$.  In Section \ref{sect:amincollection} we construct a faithful permutation representation of $\GL_n(\F_q)$.  Section \ref{sec:generallemmas} collects a number of lemmas that are used in the next two sections. We study minimal faithful permutation representations of $\GL_2(\F_q)$ in \S\ref{sec:mingl2} and $\GL_3(\F_q)$ in \S\ref{sec:mingl3} by trying to beat the faithful permutation representation construction in \S\ref{sec:generallemmas}.  We end the paper with some general comments and future directions.

\

The second author is partially supported by a Collaboration Grant (Award number 523444) from the Simons Foundation.  This paper owes a great deal of intellectual debt to the papers\cite{sl2, gl2}. We thank Lior Silberman and Ben Elias for useful conversations. We wish to thank Roman Bezrukavnikov and Annette Pilkington who independently simplified our first step of the proof of Lemma \ref{lemm5.6}. We wish to thank the anonymous referee for their careful reading of the paper, We used {\tt sagemath} to carry out some of the numerical calculations and experiments used in the paper.

\subsection*{Notation}\label{note}
In this paper $\GL_n$ stands for the algebraic group of $n \times n$ matrices with non-zero determinant, and $\SL_n$ the subgroup of $\GL_n$ with elements of determinant equal to $1$.  The finite field with $q$ elements is denoted by $\F_q$. The integer $q$ is the power of a fixed prime number $p$. We fix $p, q$ throughout the paper. 
Write $g:= \gcd(n,q-1)$ and let the prime factorization of $g$ be as follows, $g = p_1^{r_1} p_2^{r_2} \cdots p_s ^{r_s}$. Given a natural number $m$, we can write $m = p_1^{a_1} p_2^{a_2}\cdots p_s^{a_s}q_{1}^{b_{1}}\cdots q_{t}^{b_t}$, where the $p_i$'s and $q_j$'s are distinct. We set $m_{n,q} := p_1^{a_1} p_2^{a_2}\cdots p_s^{a_s}$ and $T_{n,q}(m) = \sum_{i = 1}^{r}q_r^{b_r}$.
For example, if $n= 3$, then $g:= \gcd(3,q-1)$  $ = 1$ or $3$. If $3\mid m$ and $g = 3$, then $m_{3,q}:= $ the highest power of $3$ dividing m and $T_{3,q}(m)$ is the sum of all the primes in the prime factorization of $m$ except for $3$, along with their appropriate powers. If $g = 1$, then  $m_{3,q}$:=  $1$ and $T_{3,q}(m)$ is the sum of all the primes in the prime factorization of $m$. 
\\We also introduce notation for some subgroups of $\glq$ that will be used throughout the paper.
Let 
\[ D_t =
  \{ \begin{pmatrix}
    a & 0 & \dots & 0 \\
    0 & 1 & \dots & 0 \\
    \vdots & \vdots & \ddots & \vdots \\
    0 & 0 & \dots & 1
  \end{pmatrix} \mid a \in A_t\}
\]
and
\[ Z_t =
 \{ \begin{pmatrix}
    a & 0 & \dots & 0 \\
    0 & a & \dots & 0 \\
    \vdots & \vdots & \ddots & \vdots \\
    0 & 0 & \dots & a
  \end{pmatrix}
\mid a \in A_t \}, 
\] where recall that $A_t$ is the unique subgroup of $\F_q^\times$ of size $(q-1)/t$. 
We usually denote $Z_1$ by $Z$. Note that any subgroup of $Z$ is of the form $Z_t$ for some $t \mid q-1$. In fact, $Z_t$ is the unique subgroup of $Z$ of order $(q-1)/t$. For $s, t$ divisors 
of $q-1$ we have 
\begin{equation}\label{capst}
Z_s \cap Z_t = Z_{\lcm(s, t)}. 
\end{equation}

Let 
$$
\GL_n(\F_q)^t = \{ g \in \GL_n(\F_q) \mid \det g \in A_t \}. 
$$
Then it is clear that $D_t$, $Z_t$, and $\GL_n(\F_q)^t$ are subgroups of $\GL_n(\F_q)$, and that $\GL_n(\F_q)^t = D_t \cdot \SL_n(\F_q)$.  We note that for $t \mid q-1$, 
\begin{equation}\label{indext}
[\glq: \glq^t] =t. 
\end{equation}

We define $P$ to be the set of all matrices 
$$\begin{pmatrix}
  B & x\\ 
  0 & b
\end{pmatrix}$$
where $B \in \GL_{n-1}(\F_q)$, $b^{-1} =  \det B \in \F_q^{\times}$, $x \in \F_q^{n-1}$
and $Q$ to be the set of all matrices
$$\begin{pmatrix}
  C & x\\ 
  0 & c
\end{pmatrix}$$
where, $c^{-1} =  \det(C) \in A_{(q-1)_{n,q}}$,  $x \in \F_q^{n-1}$, $C \in \GL_{n-1}(\F_q)
$ and $B,C$, and $x$ arbitrary. 
By Patton \cite{Patton}, $P$ is the subgroup of $\SL_n(\F_q)$ of minimal index. 

The standard reference for minimal faithful permutation representations of finite groups is Johnson's classical paper \cite{Johnson}. In order to construct a faithful permutation representation of a group $G$ we need to construct a collection of 
subgroups $\{H_1, \dots, H_{\ell}\}$ such that $\core_G(H_1 \cap \dots \cap H_{\ell}) = \{e \}$.  Recall that for a subgroup $H$ of $G$, $\core_G(H)$ is the largest normal subgroup of $G$ contained in $H$, i.e., 
$$
\core_G(H) = \bigcap_{x \in G} x H x^{-1}. 
$$ 
We call a collection $\{H_1, \dots, H_{\ell}\}$ of subgroups of $G$ {\em faithful} if $\core_G(H_1 \cap \dots \cap H_{\ell}) = \{e\}$. In this case the left action of $G$ on the disjoint union 
$$
A = G/H_1 \cup \dots \cup G/H_{\ell}  
$$
is faithful. Note that $|A| = \sum_i |G/H_i|$.   A collection $\{H_1, \dots, H_{\ell}\}$ is called {\em minimal faithful} if 
\begin{enumerate}
\item $\core_G(H_1 \cap \dots \cap H_{\ell}) = \{e\}$, 
\item $\sum_i |G/H_i|$ is minimal among all collections of subsets satisfying (1).   In this case, $\sum_i |G/H_i|$ is denoted by $p(G)$. 
\end{enumerate}
The papers \cite{Johnson, EST} and the thesis \cite{Elias} contain many examples of explicit computations of $p(G)$ for various groups $G$. 

\section{Minimal faithful permutation representations of $\SL_n(\F_q)$}\label{sec:sln}

In this section we study the case of $\SL_n(\F_q)$ for $q$ odd and $n \geq 2$. 

\subsection{The case of $\SL_2(\F_q)$}\label{subsec:sl2} 
We recall the construction for $\SL_2(\F_q)$ for odd $q$ from \cite{sl2}.  Write $q-1 = 2^r \cdot m$ with $m$ odd.  Let $\varpi$ be a generator of the cyclic group $\F_q^\times$. For $t \mid q-1$, set $A_t = \langle \varpi^t \rangle$. The set $A_t$ is the unique subgroup of $\F_q^\times$ of size $(q-1)/t$, as introduced in \S\ref{note}.
Note that if $s, t$ are divisors of $q-1$, then $A_s \cap A_t = A_{\lcm(s, t)}$. 
Set \begin{equation}\label{HODD}
H_{odd} = \{ \begin{pmatrix} a \\ & a^{-1} \end{pmatrix} \begin{pmatrix} 1 & x \\ & 1 \end{pmatrix} \mid a \in A_{2^r}, x \in \F_q \}. 
\end{equation}
Then Theorem 3.6 of \cite{sl2} says that $H_{odd}$ is a core free subgroup of $\SL_2(\F_q)$ of minimal index, i.e., the action of $\SL_2(\F_q)$ on $\SL_2(\F_q)/H_{odd}$ is a minimal faithful representation. Also, 
it is easy to see that 
$$
[\SL_2(\F_q): H_{odd}] = (q-1)_2 (q+1). 
$$

\subsection{The case of $\SL_3(\F_q)$} 
In this section, we compute the minimal faithful permutation representation of $\SL_3(\F_q)$  for $q$ an odd prime power. In \cite{Ash} Aschbacher classified maximal subgroups of finite classical groups. He shows that each such maximal subgroup lies in one of the 8 classes of subgroups $\mathscr{C}_1,\cdots, \mathscr{C}_8$ or its socle is a non-abelian simple group and it belongs to a special class $\mathscr{S}$. Following this classification the maximal subgroups of $\SL_3(\F_q)$ have been compiled in Tables $8.3$ and $8.4$ of \cite{sln}. Using this we compute the order of the maximal subgroups of $\SL_3(\F_q)$ and conclude the result below.
\begin{theo}\label{sl3}
The minimal faithful permutation representation of $\SL_3(\F_q)$ for $q$ an odd prime power depends on $g =  \gcd(q-1,3)$. There are two cases.

\ 

\begin{enumerate}
\item If $g = 1$, then the the subgroup $M = E_q^2 :\GL_2(\F_q)$ of class $\mathscr{C}_1$  gives the minimal faithful permutation representation of $\SL_3(\F_q)$ and $p(\SL_3(\F_q) =$  $\frac{|\SL_3(\F_q)|}{|E_q^2 :\GL_2(\F_q)|} = \frac{q^3 -1}{q-1}$.
\item If $g = 3$, then
the maximal subgroup $G_3$ of  $M$ with trivial intersection with the order $3$ subgroup of the center and whose order is not coprime to $3$ gives the minimal faithful permutation representation of $\SL_3(\F_q)$ and $p(\SL_3(\F_q)) =$  $\frac{|\SL_3(\F_q)|}{|G_{3}|}  = \frac{(q^3-1)(q-1)_3}{q-1}$.
\end{enumerate}
\end{theo}
\begin{proof} After case by case analysis, we deduce that the subgroup of $\SL_3$ with maximal order is $E_q^2 :\GL_2(\F_q)$. 
\\ Case $1$: When $g=1$ this maximal subgroup has trivial core and gives the minimal faithful permutation representation. 
\\ Case $2$:
When  $g = 3$ i.e. $q \equiv 1$ mod 3, then each maximal subgroup of $\SL_3(\F_q)$ has order not coprime to 3 and the maximal subgroup $M = E_q^2 :\GL_2(\F_q)$ has non-trivial core.
 Elements of $M$ have the form:
$\begin{pmatrix}
 a & b & *\\
c &  d & *\\
0 & 0 & \det \gamma^{-1}
\end{pmatrix}$ where $\gamma = 
\begin{pmatrix}
a & b\\
c & d
\end{pmatrix}.$
 By the argument Theorem \ref{slnmainres} the minimal faithful permutation representation in this case is offered by the maximal subgroup $G_3$ of  $M$ with trivial intersection with the order $3$ subgroup of the center and whose order is not coprime to $3$, as long as $n = 3$ is \textit{very divisible} in the sense of Definition \ref{def:div} below.

\ 

We claim that whenever $g = 3$ 
the subgroup $G_3$ gives the minimal faithful permutation representation of $\SL_3(\F_q)$ for all values of $q$ i.e. even when $3$ is not \textit{very divisible}. 
We reason as follows: the order of $G_3$ is given by $\frac{(q^2-1)(q^2-q)q^2}{(q-1)_3}$.
The maximal core-free subgroup  we want is this maximal core free subgroup $G_3$ of $M$, as the order of each maximal subgroup of $\SL_3(\F_q)$ belonging to classes outside of $\mathscr{C}_1$ is verified to be smaller than that of $G_3$.
The result follows immediately.
\end{proof}
\subsection{The case where $n\geq 4$} 
In order to state our theorem we need a definition. 
\begin{defi}\label{def:div}
Suppose $g = p_1^{r_1}\cdots p_s^{r_s}$ and for each $j$ between $1$ to $s$ let $p_j^{a_j}$ be the highest power of $p_j$ dividing $q-1$. 
If $n$ is such that $p_j^{a_j}$ divides $n$ for $1 \leq j \leq s$, then we call $n$ \textit{very divisible relative to $q$}. We usually drop {\em relative to $q$}.  For example, $n$ is \textit{very divisible} whenever $q-1$ divides $n$.
\end{defi}

\begin{theo}\label{slnmainres}
 For \textit{very divisible} $n \geq 4$ the minimal faithful representation of $\SL_n(\F_q)$ is computed as follows:
 \\If we have that $g = 1$, then the maximal subgroup subgroup $P$ of class $\mathscr{C}_1$  gives the minimal faithful permutation representation of $\SL_n(\F_q)$ and $p(\SL_n(\F_q) =$  $\frac{|\SL_n(\F_q)|}{|P|} = \frac{q^n -1}{q-1}$ by \cite{Patton}.
When $g >1$ we have that the minimal faithful permutation representation of $\SL_n(\F_q)$ is given by subgroups $H_{1}, \cdots, H_s$ and has size $p(\slq) = \frac{q^n-1}{q-1}\dot (p_1^{a_1} + \cdots + p_s^{a_s})$. Further, each $H_j$ 
is the  biggest  subgroup  of P with  trivial  intersection  with  the  order $p_j$ central  subgroup. 
\end{theo}
\begin{proof}
If we have that $g =1$, then the maximal subgroup $P$ of $\slq$ is core free and gives the minimal faithful permutation representation by \cite{Patton}. In this case, we have $p(\SL_n(\F_q) =$  $\frac{|\SL_n(\F_q)|}{|P|} = \frac{q^n -1}{q-1}$.
\\ Hence, for the remainder of the proof we will assume $g > 1$.
Suppose, $\{H_1, \dots, H_\ell\}$ was a minimal faithful representation of $\slq$. Then,  $\core_G(H_1 \cap \dots \cap H_{\ell}) = \{e \}$.We consider two cases based on order considerations of the groups of our minimal faithful collection.
\\ Case $1$): No subgroup $H_i$ has order coprime to $g =p_1^{r_1}\cdots p_s^{r_s} $, so each $H_i$ has non trivial intersection with the order $p_k$ subgroup of the center for some $k \in \{1,\cdots,s\}$. However, $\core_G(H_1 \cap \dots \cap H_{\ell}) = \{e \}$ implies that for each prime factor $p_j$ dividing $g$ there exists a subgroup in our collection $H_{i_j}$ with trivial intersection with the order $p_j$ subgroup  of the center. If this was not true, then  $\core_G(H_1 \cap \dots \cap H_{\ell})$ would contain an element of order $p_j$. Write $G_j$ as the maximal subgroup of $P$ with trivial intersection with the order $p_j$ subgroup  of the center and whose order is not coprime to $g$. We will show $H_{i_j} = G_j$. Write $R$ as some subgroup  of $\slq$ with trivial intersection with the order $p_j$ subgroup of the center and whose order is not coprime to $g$. We will show $|R| \leq |G_j|$ for $1 \leq j \leq k$, so that $H_{i_j} = G_j$. Observe that $G_j$ being a subgroup of $P$ is the set of all matrices 
$$\begin{pmatrix}
  B_j & x\\ 
  0 & b_j
\end{pmatrix}$$
where, $b_j^{-1} =  \det(B_j) \in \F_q^{\times}$, $x \in \F_q^{n-1}, B_j$ is a subgroup of $ \GL_{n-1}(\F_q)
$ and $a,x$ arbitrary.
Moreover, since $G_j$ is the biggest subgroup of $P$ with trivial intersection with the order $p_j$ central subgroup,  $B_j$ must be the largest subgroup of  $\GL_{n-1}(\F_q)$, which has trivial intersection with the order $p_j$ subgroup of the center. Since $p_j $ divides $g = \gcd (q-1,n)$,  $\SL_{n-1}(\F_q)$ will not contain a central subgroup of order $p_j$. As a central element 
 $\begin{psmallmatrix}
    a & 0 & \dots & 0 \\
    \vdots & \vdots & \ddots & \vdots \\
    0 & 0 & \dots & a
  \end{psmallmatrix}$ in $\SL_{n-1}(\F_q)$ will satisfy $a^{q-1} = a^{n-1} = 1$, then  $|a| \mid \gcd(q-1,n-1)$. So, if $|a| | \gcd(q-1,n-1)$ and $\gcd(q-1,n)$, then $|a||1$ and consequently $a = 1$. Hence, we may assume that $B_j$ contains $\SL_{n-1}(\F_q)$. By lemma \ref{subgrpscontainsln},  $B_j$ has the form $\GL_{n-1}(\F_q)^t = \{ g \in \GL_{n-1}(\F_q)  \mid \det g \in A_t$\}. Also, by  Lemma (\ref{core}) $Z \cap \GL_{n-1}(\F_q)^t = Z_{\frac{t}{\gcd(n-1,t)}}$. Recalling order of $Z_t$ is $\frac{q-1}{t}$, we require $p_j \nmid \frac{\gcd(n-1,t)(q-1)}{t}$. Also $p_j |n$
 implies $p_j \nmid n-1$, so $p_j \nmid \gcd(n-1,t)$. Thus,  $p_j \nmid \frac{q-1}{t}$. We want $B_j$ to be the largest 
 possible size, which forces $t$ to be the smallest such that $p_j \nmid \frac{q-1}{t}$. Hence, $t = (q-1)_{p_j} = p_j^{a_j}$, that is the highest power of $p_j$ dividing $q-1$. Hence, $B_j = \GL_{n-1}(\F_q)^{p_j^{a_j}} = D_{p_j^{a_j}}. \SL_{n-1}(\F_q) $ and $\GL_{n-1}(\F_q) =   D_1. \SL_{n-1}(\F_q)$. Further,
 $D_1 = D_{\frac{q-1}{t}}.D_{t}$ gives $\GL_{n-1}(\F_q) = D_{\frac{q-1}{p_j^{a_j}}}. B_j$. Writing $C_j = D_{\frac{q-1}{p_j^{a_j}}}$ we have $\GL_{n-1}(\F_q) = C_j.B_j$. Let $P_j$ be the  set of all matrices
$$\begin{pmatrix}
  C_j & 0\\ 
  0 & c_j
\end{pmatrix}$$
where, $c_j^{-1} =  \det(C_j) \in \F_q^{\times}$, $x \in V_{n-1}, C_j$ as defined above,  $a,x$ arbitrary. Hence, $P = G_j.P_j$. Observe that 
$|P_j| = |C_j| = |D_{\frac{q-1}{p_j^{a_j}}}| = p_j^{a_j} = |Z_{\frac{q-1}{p_j^{a_j}}}|$. Thus, $R.Z_{\frac{q-1}{p_j^{a_j}}} = Z_{\frac{q-1}{p_j^{a_j}}}.R$
 is a subgroup of $\glq$. If $p_j^{a_j}|n$ for all $j$, then $Z_{\frac{q-1}{p_j^{a_j}}}$ is a subgroup of $\slq$ and so is $R.Z_{\frac{q-1}{p_j^{a_j}}}$. 
 But, $P$ is the largest subgroup of $\slq$, hence index  of $P$ in $\slq = \frac{|\slq|}{|G_j|.|P_j|} \leq \frac{|\slq|}{|R|.|Z_{\frac{q-1}{p_j^{a_j}}}|}$. 
 Using that $|Z_{\frac{q-1}{p_j^{a_j}}}| = |P_j|$ and simplifying we obtain, $|R| \leq |G_j|$ as claimed. Hence, $H_{i_j} = G_j$ and for the collection 
 to be minimal it must consist of $H_{i_1},\cdots,H_{i_s}$ with degree $[G:G_1] + \cdots + [G:G_s] = [G:P]\cdot (|P_1| + \cdots + |P_s|) = 
 \frac{q^n-1}{q-1}\cdot(p_1^{a_1} + \cdots p_s^{a_s})$.
\\ Case $2$): There is some $i$, such that $1 \leq i \leq \ell$ and $\gcd(|H_i|,g) = 1$.  We recall from Lemma $3.2$ of \cite{sl2}, that the only proper normal subgroups of $\slq$ are central subgroups. Using this it follows that any such subgroup $H_i$ with $\gcd(|H_i|,g) = 1$ has trivial core, because any element of the center has order dividing $g$. The intersection of any such subgroups $H_i$ will consequently have trivial core as well.  Labeling the groups in our minimal faithful collection that have order coprime to $g$ with a $'$ mark, that is  $H'_i$ and labeling the other subgroups $H_i$ in our collection without the $'$ mark we obtain the inequality $\Sigma_i [G:H'_i] + \Sigma_j [G:H_j] > \Sigma_i [G:H'_i]$. Hence, the collection with only the  subgroups $\{H_i'\}$ for some indexing set $i \in I$ not only has trivial core thus forming a faithful collection, but also it has lower degree than the minimal faithful collection comprising all the subgroups $H_i$. This is a contradiction, unless our minimal faithful collection has all subgroups of the form $H_i'$, that is each subgroup in our collection has order coprime to $g$. However, in this case we recall  again that any such subgroup has trivial intersection with the center $Z(\slq)$ and thus it has trivial core. Thus, we could replace the collection of subgroups $H_i'$ with a single subgroup $H$ of maximal order  coprime to $g$. This would give a faithful collection with order lesser than the minimal one, forcing our minimal faithful collection to comprise of a single subgroup $H$ satisfying $(|H|, g) = 1$.
Next, following the proof of Lemma $3.5$ in \cite{sl2} we write $S$ to denote $Z(\slq)$ to observe that $S \cap H = \{1\}$ and $SH/S \subset G/S$. This gives us that  $ SH/S \cong H/S \cap H \cong H  \subset G/S \cong \pslq$. In other words subgroups H of G of order coprime to g can be identified with subgroups of $\pslq$ of order coprime to  $g$. However, we know from Patton's result in \cite{Patton} that $\rho(\pslq)$ is $(q^n-1)/(q-1)$.  Combining all this information with the fact that $\pslq = \slq / Z(\slq)$ and $|\pslq| = |\slq| / g$ we have that $p(\slq) = |G|/|H| = |\pslq|g/|H|$ is at least $\rho(\pslq)g = (q^n-1)/(q-1) \times g$.  However, recalling the prime factorization of $g$ as $\Pi_jp_j^{a_j}$ we obtain $\rho(\pslq)g \geq (q^n-1)/(q-1) \times (\Pi_jp_j^{a_j}) >  (q^n-1)/(q-1) \times (\Sigma_jp_j^{a_j}) $ by Lemma \ref{lemm33}, which is the index in the previous case. Since the size of the faithful collection in case $1$ is smaller than the one provided by the subgroup $H$ in this case, we conclude that the minimal faithful permutation representation is always provided by the subgroups in case $1$.
\end{proof}

\begin{exam}\label{ex:special}
For example, if $n = 4$ and $q = 41$, then $g = \gcd(4,40) = 4$ and $n$ is not a \textit{very divisible} natural number. Utilizing the classification in Table $8.8$ of \cite{sln} we deduce that the maximal subgroup of $\SL_4(\F_{41})$ is given by $E_{41}^3:\GL_3(41)$ i.e. the subgroup whose elements have the form:
$\begin{pmatrix}
a & b & c & *\\
d & e & f & *\\
g & h & i & *\\
0 & 0 & 0 & \det \gamma^{-1}
\end{pmatrix} $
where $\gamma$ is the matrix
$
\begin{pmatrix}
a & b & c\\
d & e & f\\
g & h & i
\end{pmatrix}.$ This subgroup has index
$\frac{q^4 -1}{q -1}$. 
This subgroup has trivial core when $g=1$, but in our case $g$ is $4$. Theorem \ref{slnmainres} implies that 
the minimal faithful permutation representation of $\SL_4(\F_{41})$ is provided by the maximal subgroup $G_4$ of  $E_{41}^3:\GL_3(41)$ with trivial core, as long as $n$ is \textit{very divisible}. In our case $n$ is not \textit{very divisible}, as $40$ is divisible by $2^3$, while $n = 4$ is not. We still claim that the minimal faithful permutation representation of $\SL_4(\F_{41})$ is provided by the maximal subgroup $G_4$ of  $E_{41}^3:\GL_3(41)$ with order coprime to $4$ in this case. By Theorem \ref{slnmainres}, $G_4$ has index $\frac{(q^4-1)(q-1)_2}{(q-1)}$ and its order can be verified to be larger than the order of each maximal subgroup of $\SL_4(\F_{41})$ and thus it is the required subgroup.
\end{exam}

This examples and others like it support the following conjecture. 

\begin{conj}\label{conj:sln}
For any $n$, Theorem \ref{slnmainres} holds true for all $\SL_n(\F_q)$ for $q$ an odd prime power. 
\end{conj}

\section{A faithful collection for $\glq$}\label{sect:amincollection}

In this section we construct a faithful collection for $\glq$.

The following lemma is a consequence 
of the Lattice Isomorphism Theorem: 

\begin{lemm}\label{subgrpscontainsln}
If $H$ is a subgroup of $\glq$ which contains $\slq$, then there is  $t \mid q-1$ such that $H = \glq^t$. 
\end{lemm}

The following lemma is important:
\begin{lemm}\label{core}
$Z \cap \GL_n(\F_q)^t = Z_\frac{t}{\gcd(n,t)}$.
\end{lemm}

\begin{proof}
 $Z \cap \glq^t$ is a subgroup of $Z$, so it must be of the form $Z_s$ for some $s \mid (q-1)$. 
We need to determine the diagonal elements of the form  $$\begin{pmatrix}
    a & 0 & \dots & 0 \\
    0 & a & \dots & 0 \\
    \vdots & \vdots & \ddots & \vdots \\
    0 & 0 & \dots & a
  \end{pmatrix}
$$ that can be written in the form 
$$\begin{pmatrix}
    \varpi^{kt} & 0 & \dots & 0 \\
    0 & 1 & \dots & 0 \\
    \vdots & \vdots & \ddots & \vdots \\
    0 & 0 & \dots & 1
  \end{pmatrix}
  \begin{pmatrix}
    a_1 & 0 & \dots & 0 \\
    0 & a_2 & \dots & 0 \\
    \vdots & \vdots & \ddots & \vdots \\
    0 & 0 & \dots & a_n
  \end{pmatrix}
$$
for some integer  $0 \leq k < (q-1)/t$, such that  $a_1,a_2,\cdots,a_n  \in \F_q ^{\times}$ and  $a_1.a_2\cdots a_n =1$ Observe that, 
$\varpi^{kt}a_1 = a_2 = \cdots = a_n = a$ gives us $  a^{n}(\varpi^{kt})^{-1} = 1$, which means $ \varpi^{kt} = a^n$.  
Letting $a = \varpi^{\ell}$ we obtain $\varpi^{kt} = \varpi^{n\ell}$. This is true if and only if  $kt \equiv n\ell$ mod $q-1$. 
There will be a solution for $\ell$ if and only if $\gcd(n,q-1)$ divides $kt$ if and only if $\frac{\gcd(n,q-1)}{\gcd(n,q-1,t)}$ 
divides $\frac{t}{\gcd(n,q-1,t)}k$. But, $\frac{\gcd(n,q-1)}{\gcd(n,q-1,t)}$ and $\frac{t}{\gcd(n,q-1,t)}$ are coprime. Hence, 
$\gcd(n,q-1)$ divides $kt$ if and only if $\frac{\gcd(n,q-1)}{\gcd(n,q-1,t)}$ 
divides k. But, $t |q-1$ implies $\gcd(n,q-1,t) = \gcd(n,t)$. Combining everything  $\varpi^{kt} = \varpi^{n\ell}$ has a solution for 
$\ell$ iff. $\frac{\gcd(n,q-1)}{\gcd(n,t)}$ divides k. Observe $0 \leq k < \frac{q-1}{t}$. Hence, the number of possibilities for
 $k$ are $\frac{\frac{q-1}{t}}{\frac{\gcd(n,q-1)}{\gcd(n,t)}} = \frac{(q-1)\gcd(n,t)}{t\gcd(n,q-1)}$.

Dividing $kt \equiv n\ell$ mod $q-1$ throughout by $\gcd(n,q-1)$ we obtain
\begin{equation*}
    \ell \equiv  \big(\frac{n}{\gcd(n,q-1)} \big)^{-1} \frac{kt}{\gcd(n,q-1)} \mod \frac{q-1}{\gcd(n,q-1)}
\end{equation*}
Hence, for each value of $k$ there will be $\gcd(n,q-1)$ corresponding values for $\ell$. Thus, the total number of possibilities for $a$ i.e. for $k\ell$ are  $\frac{(q-1)\gcd(n,t)}{t\gcd(n,q-1)}. \gcd(n,q-1) = \frac{(q-1)\gcd(n,t)}{t}$. Thus, we conclude that $Z \cap \GL_n(\F_q)^t = Z_\frac{t}{\gcd(n,t)}$ as claimed.
\end{proof}
To state our results we need a couple of pieces of notation. 
Let $P$,$Q$, $m_{n,q}$, and $T_{n,q}$ be defined in the same manner as in  \S\ref{note}.
\begin{lemm}\label{lemm04}
We have 
$$
\core_{\glq}(Q\cdot D_1) = Z_{(q-1)_{n,q}} \text{ and } \  \core_{\glq}{P\cdot D_1} = Z. 
$$
\end{lemm}
\begin{proof}
Observe that $\SL_n(\F_q)$ is not contained in $Q.D_1$. So, the core must be the intersection of $Q.D_1$ with the center $Z$ of $\GL_n(\F_q)$. For  an element of $Q.D_1$ to be in the center we require the element
$ D = \begin{pmatrix}
  C & x\\ 
  0 & c
\end{pmatrix}$ 
in $Q$ to be diagonal. 
In particular, If $D =  $diag($c_{1}, \cdots, c_{n-1}$), then multiplying this by an element diag($b,1,\cdots,1$) of $D_1$ should give us a diagonal matrix. That is, $c_{1}b = c_{2} = \cdots = c_{n-1} = c$ and using $c^{-1} = $ det$C$ we have the equation $c_1c^{n-2} = c^{-1}$ i.e. $c^{n-1}b^{-1} = c^{-1}$ or $c^n = b$. Note that $b$ is any element in $\F_q^*$, so as long as $c \neq 0$  pick $b = c^n$ and get the element diag$(c,c,\cdots,c$) in $Q.D_1 \cap Z$. So, the core of $Q.D_1$ is $Z_{(q-1)_{n,q}}$.
By similar considerations we obtain $\core_{\glq}{P\cdot D_1} = Z$.
\end{proof}
We can now give a construction of a faithful permutation representation which we will later prove to be minimal for $n = 2$ and $3$. 
\begin{prop}\label{fcoll}
Write $q-1 = p_1^{a_1} p_2^{a_2}\cdots p_s^{a_s}q_{1}^{b_{1}}\cdots q_{t}^{b_t}$, with $g = \gcd(n,q-1) = p_1^{r_1} p_2^{r_2} \cdots p_s ^{r_s}$, $r_1, \dots, r_s \geq 1$ and the $p_i's$ and $q_i's$ distinct. Then the collection of subgroups
$$
\{ Q\cdot D_1, \glq^{q_1^{b_1}}, \dots, \glq^{q_t^{b_t}} \}
$$
is a faithful collection. The corresponding faithful permutation representation has size 
$$\frac{q^n-1}{q-1}\cdot (q-1)_{n,q} + T_{n,q}(q-1).$$ 
If $g = 1$ then the collection of subgroups
$$
\{ P\cdot D_1, \glq^{q_1^{b_1}}, \dots, \glq^{q_t^{b_t}} \}
$$
is a faithful collection. The corresponding faithful permutation representation has size $$\frac{q^n-1}{q-1} + T_{n,q}(q-1).$$ 
\end{prop}
\begin{proof}
We need to show that the subgroup $ U = Q\cdot D_1 \cap \glq^{q_1^{b_1}} \cap \dots \cap \glq^{q_t^{b_t}}$ is core-free. By Lemma \ref{lemm04}, $\core_\glq(Q \cdot D_1)$ is a subset of $Z$, 
so the core of $U$ is a subgroup of $Z$. Since any subgroup of $Z$ is normal, we just need to compute the intersection $U \cap Z$. Lemmas \ref{core} and \ref{lemm04} implies that 
this intersection is $Z_{(q-1)_{n,q}} \cap Z_{q_1^{b_1}} \cap \dots Z_{q_t^{b_t}}$. As $\lcm((q-1)_{n,q}, q_1^{b_1}, \dots, q_t^{b_t}) = q-1$, Equation \eqref{capst}  says 
$$
Z_{(q-1)_{n,q}} \cap Z_{q_1^{b_1}} \cap \dots Z_{q_t^{b_t}} = Z_{q-1} = \{e\}. 
$$
This means that the collection is faithful. The size of the corresponding permutation representation is equal to 
$$
[\glq: Q\cdot D_1] + [\glq: \glq^{q_1^{b_1}} ] + \dots + [\glq: \glq^{q_t^{b_t}} ]
$$
$$
= \frac{q^n-1}{q-1}\cdot (q-1)_{n,q} + q_1^{b_1} + \dots + q_t^{b_t}
$$
from \cite{Patton} and \eqref{indext}. 

By symmetry, when $g = 1$, we need to show that the subgroup $ U' = P\cdot D_1 \cap \glq^{q_1^{b_1}} \cap \dots \cap \glq^{q_t^{b_t}}$ is core-free. By Lemma \ref{lemm04}, $\core_\glq(P \cdot D_1)$ is  $Z$, so the core of $U$ is a subgroup of $Z$. Since any subgroup of $Z$ is normal, we just need to compute the intersection $U \cap Z$. Lemmas \ref{core} and \ref{lemm04} implies that 
this intersection is $Z \cap Z_{q_1^{b_1}} \cap \dots Z_{q_t^{b_t}}$. As $\lcm(1, q_1^{b_1}, \dots, q_t^{b_t}) = q-1$, Equation \eqref{capst}  says 
$$
Z \cap Z_{q_1^{b_1}} \cap \dots Z_{q_t^{b_t}} = Z_{q-1} = \{e\}. 
$$
This means that the collection is faithful. The size of the corresponding permutation representation is equal to 
$$
[\glq: P\cdot D_1] + [\glq: \glq^{q_1^{b_1}} ] + \dots + [\glq: \glq^{q_t^{b_t}} ]
$$
$$
= \frac{q^n-1}{q-1}+ q_1^{b_1} + \dots + q_t^{b_t}
$$
after using Equations \cite{Patton} and \eqref{indext}. This finishes the proof of the proposition. 
\end{proof}

\begin{coro}\label{coro25}
We have 
$$
p(\glq) \leq \frac{q^n-1}{q-1}\cdot (q-1)_{n,q} + T_{n,q}(q-1). 
$$
\end{coro}

We make the following conjecture: 

\begin{conj}\label{conj11}
Let $q \geq 5$ be an odd prime power.  Then if $g=1$, 
$$p(\GL_n(\F_q)) = \frac{q^n-1}{q-1} + T_{n,q}( q-1),$$
whereas if $g > 1$, then 
$$p(\GL_n(\F_q)) = p(\SL_n(\F_q)) + T_{n,q}( q-1). $$
\end{conj}

We will verify this conjecture for $n=2$ and $n=3$ in \S \ref{sec:mingl2}  and \S\ref{sec:mingl3}, respectively.  That the conjecture is true for very divisible $n$ is a consequence of Lemma \ref{lemma35} below. 

\section{Minimal faithful collections} \label{sec:generallemmas}
In this section we will prove some general lemmas about minimal faithful collections that will help us determine these collections for $n = 2$ and $3$ in the next two sections.
\begin{lemm}\label{lemm31}
Let $H$ be a subgroup of $\glq$. Then 
$$
[H: H\cap \slq] = \frac{|H \cdot \slq|}{|\glq|}\cdot (q-1). 
$$
\end{lemm}
\begin{proof}
We observe that $H\cdot \slq$ is a subgroup of $\glq$. We have 
$$
|H \cdot \slq| = \frac{|H| \cdot |\slq|}{|H \cap \slq|}.  
$$
Hence, 
$$
[H: H\cap \slq] = \frac{|H|}{|H \cap \slq|}
$$
$$
 = \frac{|\glq|}{[\glq: H \cdot \slq]} \cdot \frac{1}{|\slq|}
$$
$$
=  \frac{q-1}{[\glq: H \cdot \slq]}, 
$$
as claimed. 
\end{proof}
\begin{lemm}\label{lemm32}
Let $H$ be a subgroup of $\glq$ such that $H \cap \slq$ is a  core free subgroup of $\slq$. Then,
$$
[\glq: H] \geq p(\slq) \cdot [\glq: H \cdot \slq]. 
$$
\end{lemm}
\begin{proof}
Since 
$H \cap \slq$ is a core free subgroup of $\slq$ we have the inequality
\begin{equation}\label{eq31}
[\slq: H \cap \slq] \geq p(\slq). 
\end{equation}
Next, 
$$
[\glq:H] = \frac{[\glq: H \cap \slq]}{[H: H \cap \slq]}. 
$$
By Lemma \ref{lemm31} this expression is equal to 
$$
= \frac{[\glq: H \cap \slq]}{q-1} \cdot [\glq: H \cdot \slq]
$$
$$
= [\slq: H \cap \slq] \cdot [\glq: H \cdot \slq]
$$
$$
\geq p(\slq) \cdot  [\glq: H \cdot \slq], 
$$
\text{by Equation} \eqref{eq31}. 
\end{proof}
\begin{lemm}\label{lemm33}
For natural numbers $a_1, \dots, a_k \geq 2$, at least one of which is strictly larger than $2$,  we have 
$$
\sum_i a_i < \prod_i a_i. 
$$ 
\end{lemm}
\begin{proof}
Proof is by induction, without loss of generality assume $a_1 \geq 3$. We have 
$$
(a_1 - 1) \cdot (a_2 -1) \geq 2. 
$$
Simplifying gives $a_1 a_2 \geq a_1 + a_2 + 1$. The rest is clear. 
\end{proof}

\begin{coro}\label{coro34}
We have 
$$
T_{n,q}(q-1) < p(\slq). 
$$ when $n = 2$ or $3$ or $n > 3$ is a very divisible integer.
\end{coro}
\begin{proof} 
If $n = 2$, then by lemma \ref{lemm33} 
$$
T_{2,q}(q-1) \leq \frac{q-1}{(q-1)_{2,q}} < q-1 < q+1 < (q-1)_{2,q} (q+1) = p(\SL_2(\F_q))
$$ upon using the statements in \S\ref{subsec:sl2}. 

\

If $n=3$, we have two cases depending on whether or not $3$ divides $q-1$.

\

Suppose $3$ does not divide $q-1$, so by Theorem \ref{sl3}, $p(\SL_3(\F_q)= \frac{q^3 -1}{q-1}$.
By Lemma \ref{lemm33}, 
$$
T_{3,q}(q-1) \leq \frac{q-1}{(q-1)_{3,q}} < q-1 <
q+1 < \frac{q^3 -1}{q-1} = p(\SL_3(\F_q).
$$

Next, suppose $3$ divides $q-1$, so by Theorem \ref{sl3}, 
$p(\SL_3(\F_q)= {(q^2+q+1)(q-1)_{3,q}}{}$.
By Lemma \ref{lemm33}, 
$$
T_{3,q}(q-1) \leq \frac{q-1}{(q-1)_{3,q}} < q-1 
< p(\SL_3(\F_q).
$$
For general very divisible $n > 3$ by Lemma \ref{lemm33}  and Theorem \ref{slnmainres}, 
$$
T_{n,q}(q-1) \leq \frac{q-1}{(q-1)_{n,q}} \leq q-1
<  (q-1)\frac{(q^{n - 1} + q^{n -2} + \cdots + 1)}{q-1} = 
$$

$$ \leq
\frac{(q^n-1)}{q-1} 
\leq \frac{(q^n-1)}{q-1}(\Sigma_j p_j) \leq 
p(\SL_n(\F_q)).
$$
\end{proof}
Now let $\mathcal C = \{H_1, \dots, H_\ell\}$ be a minimal faithful collection of $\GL_n(\F_q)$ for $n$ prime.
Recall, that the collection being faithful  means that we have $\core_G(H_1 \cap \dots \cap H_{\ell}) = \{e \}$. In particular, no central subgroup of $\glq$ can lie in the intersection
of the  groups $H_i$, since the center  of $\glq$ is always a normal subgroup. Further recall that $Z(\slq) \subset Z(\glq)$, implies  that for each prime factor $p_j$ dividing $g$ there must exist a subgroup $H_{i_j}$ in our collection , such that $H_{i_j} \cap \slq$ has trivial intersection with the order $p_j$ subgroup  of $Z(\slq)$. If this was not true, then  $\core_G(H_1 \cap \dots \cap H_{\ell})$ would contain the order $p_j$ subgroup of $Z(\slq)$. Recall, that any central element of $\slq$ has order dividing $g$, since it has the form diag$(a,\cdots,a)$ with $a^n = a^{q-1} = 1$ i.e. $a^g = 1$.  Also, since we are assuming $n$ is prime  we must have that $g$ equals either $1$ or $n$, as $g$ divides $n$.
Thus, in the case of $n$ prime, we conclude that there must be some index $i$ for which the subgroup $H_i \cap \slq$ of $\slq$ does not contain the central element of  $\slq$ of order $n$, if it exists, thus making it a core free subgroup of $\slq$. 
\begin{lemm}\label{lemma35} In the case that $n = 2,3,$ or $n >3$ is  very divisible and prime, there is at most one $i$ such that $H_i \cap \slq$ is a core free subgroup of $\slq$ and $H_i \cdot \SL_{n}(\F_q) = \GL_{n}(\F_q)$. 
\end{lemm}
\begin{proof}
Suppose $H_i \cap \slq, H_j\cap \slq$ are both core free subgroups of $\slq$.
Hence,
by Lemma \ref{lemm32},
$
[\GL_{n}(\F_q): H_i ] + [\GL_{n}(\F_q):H_j] 
$ is larger than or equal to 
$$
([\GL_n(\F_q): H_i \cdot \SL_n(\F_q)] + [\GL_n(\F_q): H_j \cdot \SL_n(\F_q)]) \cdot p(\SL_n(\F_q) 
$$
which is at least $2 p(\slq)$. By Corollary \ref{coro34} we have $2 p(\SL_n(\F_q)) > p(\SL_n(\F_q)) + T_{n,q}(q-1)$, and this latter quantity, by Corollary \ref{coro25} and \ref{slnmainres}, is larger than or equal to $p(\glq)$. Consequently,
if $H_i \ne H_j$ or if $H_i \cdot \SL_n(\F_q) \ne \GL_n(\F_q)$,  
$$
[\GL_n(\F_q): H_i ] + [\GL_n(\F_q):H_j]  > p(\GL_n(\F_q)). 
$$
This contradicts the assumption that $\mathcal C$ is a minimal faithful collection. 
Without loss of generality let $i=1$. 
\end{proof}
Our goal is to minimize 
$$
[\GL_n(\F_q): H_2] + \dots + [\GL_n(\F_q): H_\ell]. 
$$
We need a lemma. 
\begin{lemm}\label{lemma39}  
Suppose $t \mid q-1$, and $t \ne q-1$. Let $H(t)$ be the subgroup of $\glq$ with minimal $[\glq: H]$ among the subgroups that satisfy $Z \cap H = Z_t$. Then
$$
H(t) = 
\glq^{dt}, 
$$ where $d$ is any divisor of $g = \gcd(n,q-1)$.
Furthermore, 
$$
[\glq: H(t)] = 
dt.
$$
\end{lemm}
\begin{proof} 
By Lemma \ref{core} there is a subgroup $H$ containing $\SL_n(\F_q)$ which satisfies the conditions of the lemma. 
In fact we will show the subgroup $H$ in the statement of the lemma must contain $\SL_n(\F_q)$.

Let $m_t$ be the minimum degree $[\glq:K]$, over subgroups $K$ of $\glq$ containing $\SL_n(\F_q)$ and $Z_t$. Now, the $H$ in the lemma satisfies $[\glq:H] = [\glq: H \cap \SL_n(\F_q)]/[H: H \cap \SL_n(\F_q)] =  [\glq: H \cap \SL_n(\F_q)].\frac{[\GL_n(\F_q): H.\SL_n(\F_q)]}{q-1}$ by Lemma \ref{lemm31}. Observe, $H.\SL_n(\F_q)$ is a subgroup of $\glq$ containing both $Z_t$(as $H$ contains it) and $\SL_n(\F_q)$. Hence,$[\glq: H] \geq [\glq: H \cap \SL_2(\F_q)].\frac{m_t}{q-1}$. We require $[\glq:H]$ to be minimal among all subgroups containing $Z_t$ by definition. So, we need to minimize $[\glq, H \cap \SL_n(\F_q)]$. For this, it's clear that $|H \cap \SL_n(\F_q)|$ must be maximized, hence $H$ contains $\SL_n(\F_q)$. So, by Lemma \ref{subgrpscontainsln} $H = \glq^{t'}$ for some $t' | q-1$. Combining Lemma \ref{core} and $H \cap Z = Z_t$, we obtain $t = \frac{t'}{\gcd(n,t')}$. However, $\gcd(n,t')|\gcd(n.q-1) = g \implies t' = dt$ for some divisor $d$ of $g$ and $H = \glq^{dt}$ as claimed. The last assertion follows from Equation\ref{indext}.
\end{proof}

\section{Minimal faithful permutation representations of $\GL_2(\F_q)$}\label{sec:mingl2}
We will determine the size and the structure of the minimal faithful permutation representations of the group $\GL_2(\F_q)$, for an odd prime power $q$.  Theorem 3.7 of \cite{gl2} claims to have determined at least the size of the minimal faithful permutation representation, but there is a typo in the answer, and it appears to us that the proof presented is not correct. The proof we present here is inspired by the results
and techniques of \cite{sl2, gl2}.  

\

We will modify the notation slightly noting that $\gcd$ of $2$ and $q-1$ is always $2$, since $q$ is odd. Given a natural number $n$, we can write $n = 2^r \prod_{p \text{ odd}} p^{e_p}$.  We set $n_{2,q} = 2^r$ and $T_{2,q}(n) = \sum_{p \text{ odd}} p^{e_p}$.
Also given a finite group  $G$ we let $p(G)$ be the size of the minimal faithful permutation representation of of $G$, i.e., the size of the smallest set $A$ on which $G$ has a faithful action. We will be proving the following 
result: 
\begin{theo}\label{theo11}
If $q \geq 3$ is an odd prime power, then 
$$
p(\GL_2(\F_q)) = p(\SL_2(\F_q)) + T_{2,q}( q-1) = (q+1)(q-1)_{2} + T_{2,q}(q-1). 
$$
\end{theo}

In fact we prove a much stronger theorem, Theorem \ref{theo313}, where we identify all minimal faithful sets for $\GL_2(\F_q)$. The equality $p(\SL_2(\F_q)) = (q+1)(q-1)_{2,q}$ is Theorem 3.6 of \cite{sl2}.   To prove our theorem we first construct a faithful permutation representation of size $p(\SL_2(\F_q)) + T_{2,q}( q-1)$ 
and then we proceed to find {\em all} minimal faithful permutation representations of $\GL_2(\F_q)$ by trying to beat this bound. The proof we present here is elementary but rather subtle. 

\

Recall the construction of the minimal faithful permutation representation of $\SL_2(\F_q)$ from \S\ref{subsec:sl2}. 
Set 
\begin{equation}\label{GHODD-D1}
GH_{odd} = D_1 \cdot H_{odd}. 
\end{equation}
Then $GH_{odd}$ is a subgroup of $\GL_2(\F_q)$, and we have 
\begin{equation}\label{GHodd}
[\GL_2(\F_q): GH_{odd}] = (q-1)_{2,q} (q+1). 
\end{equation}
\begin{lemm}\label{coreGHodd}
We have 
$$
\core_{\GL_2(\F_q)}(GH_{odd}) = Z_{2^r}. 
$$
\end{lemm}
\begin{proof}
Any normal subgroup of $\GL_2(\F_q)$ which does not contain $\SL_2(\F_q)$ is a subgroup of $Z$, so we just need to show $Z \cap GH_{odd} = Z_{2^r}$. For $a \in A_{2^r}$ we have 
$$
\begin{pmatrix} a^2 \\ & 1 \end{pmatrix} \cdot \begin{pmatrix} a^{-1} \\ & a \end{pmatrix} = \begin{pmatrix} a \\ & a \end{pmatrix}. 
$$
\end{proof}

We can now give a construction of a faithful permutation representation which we will later prove to be minimal. 

\begin{prop}
Write $q-1 = 2^r \cdot p_1^{e_1} \cdots p_k^{e_k}$ with $p_1, \dots, p_k$ distinct odd primes, and $e_1, \dots, e_k \geq 1$. Then the collection of subgroups 
$$
\{ GH_{odd}, \GL_2(\F_q)^{p_1^{e_1}}, \dots, \GL_2(\F_q)^{p_k^{e_k}} \}
$$
is a faithful collection. The corresponding faithful permutation representation has size $p(\SL_2(\F_q)) + T_{2,q}(q-1)$. 
\end{prop}
\begin{proof}
We need to show that the subgroup $ U = GH_{odd} \cap \GL_2(\F_q)^{p_1^{e_1}} \cap \dots \cap \GL_2(\F_q)^{p_k^{e_k}}$ is core free.   By Lemma \ref{coreGHodd}, $\core_{\GL_2(\F_q)}(GH_{odd})$ is a subset of $Z$, 
so the core of $U$ is a subgroup of $Z$. Since any subgroup of $Z$ is normal, we just need to compute the intersection $U \cap Z$. Lemmas \ref{core} and \ref{coreGHodd} imply that 
this intersection is $Z_{2^r} \cap Z_{p_1^{e_1}} \cap \dots Z_{p_k^{e_k}}$. As $\lcm(2^r, p_1^{e_1}, \dots, p_k^{e_k}) = q-1$, Equation \eqref{capst}  says 
$$
Z_{2^r} \cap Z_{p_1^{e_1}} \cap \dots Z_{p_k^{e_k}} = Z_{q-1} = \{e\}. 
$$
This means that the collection is faithful. The size of the corresponding permutation representation is equal to 
$$
[\GL_2(\F_q): GH_{odd}] + [\GL_2(\F_q): \GL_2(\F_q)^{p_1^{e_1}} ] + \dots + [\GL_2(\F_q): \GL_2(\F_q)^{p_k^{e_k}} ]
$$
$$
= (q-1)_{2,q}(q+1) + p_1^{e_1} + \dots + p_k^{e_k}
$$
after using Equations \eqref{GHodd} and \eqref{indext}. This finishes the proof of the proposition. 
\end{proof}

\begin{coro}\label{coro5.4}
We have 
$$
p(\GL_2(\F_q)) \leq (q-1)_{2,q}(q+1) + T_{2,q}(q-1) = p(\SL_n(\F_q)) + T_{2,q}(q-1). 
$$
\end{coro}
Without loss of generality let $i=1$. 

\begin{lemm}\label{lem5.5}
The 
$
H_1 \cap \SL_2(\F_q)  
$
is a conjugate of $H_{odd}$ in $\SL_2(\F_q)$, with $H_{odd}$ defined by Equation \eqref{HODD}. 
\end{lemm}
\begin{proof}
Since $\begin{pmatrix} -1 \\ & -1 \end{pmatrix} \not\in H_1$, $H_1 \cap \SL_2(\F_q)$ will have odd order.  By Lemma 3.5 of \cite{sl2}, up to conjugation, we have the following 
possibilities for $H_1 \cap \SL_2(\F_q)$: 
\begin{enumerate}
\item[(A)]\label{case1} a cyclic subgroup of odd order dividing $q \pm 1$; 
\item[(B)]\label{case2} a subgroup of odd order of the upper triangular matrices
$$
T(2, q) = \{ \begin{pmatrix} a \\ & a^{-1} \end{pmatrix} \begin{pmatrix} 1 & x \\ & 1 \end{pmatrix} \mid a \in \F_q^\times, x \in \F_q. \}. 
$$
\end{enumerate}
In case (A), since $|H_1 \cap \SL_2(\F_q)| = (q\pm 1)/t$ is odd, and $q \pm 1$ are even numbers,  we must have $t \geq 2$. By the proof of Lemma \ref{lemm32} and the statement of Lemma \ref{lemma35} 
$$
[\GL_2(\F_q): H_1] = [\SL_2(\F_q): H_1 \cap \SL_2(\F_q)] \cdot [\GL_2(\F_q): H_1 \cdot \SL_2(\F_q)] 
$$
$$
= [\SL_2(\F_q): H_1 \cap \SL_2(\F_q)]  = \frac{|\SL_2(\F_q)|}{|H_1 \cap \SL_2(\F_q)|}
$$
$$
= \frac{q(q+1)(q-1)}{q\pm 1} \cdot t \geq 2 \frac{q(q+1)(q-1)}{q\pm 1}.  
$$
We determine the cases where $\frac{q(q+1)(q-1)}{q\pm 1} \geq (q-1)_{2,q} \cdot (q+1)$.   We need  $\frac{q(q-1)}{q\pm 1} \geq (q-1)_{2,q}$. We have two cases: 

\begin{itemize}
\item If the denominator is $q-1$, then we need $q \geq (q-1)_{2,q}$, and that's obviously true. 
\item If the denominator is $q+1$, then we want $q(q-1) \geq (q+1) (q-1)_{2,q}$.  For this write $q-1 = 2^r m$ with $m$ odd, then we need $(2^r m + 1) 2^r m \geq (2^r m + 2) 2^r$, or 
what is the same, $(2^r m +1)m \geq 2^r m + 2$.   If $m \geq 3$, then this last inequality is definitely satisfied, but if $m =1$, it is not true. 
\end{itemize}
This discussion means that unless $q=2^r + 1$, $[\GL_2(\F_q): H_1]  \geq 2 p(\SL_2(\F_q))$ which by Corollary \ref{coro34} is strictly bigger than $p(\SL_2(\F_q)) + T_{2,q}(q-1)$. This last statement, by Corollary \ref{coro5.4}, contradicts the minimality of $\mathcal C$ 

\

Now  we examine the case where $q = 2^r +1$. Note that in this  case $T_{2,q}(q-1) =0$ as $q-1$ has on odd prime factors. One easily checks that 
$$
2 \frac{q(q+1)(q-1)}{q\pm 1} > (q-1)_{2,q}(q+1). 
$$
This means that $[\GL_2(\F_q): H_1] > p(\SL_2(\F_q)) + T_{2,q}(q-1)$ which again, by Corollary \ref{coro5.4}, contradicts the minimality of $\mathcal C$. 

\

Now we examine case (B). In this case, if we write $q-1 = 2^r m$, there is a divisor $m_0$ of $m$ such that 
$$
H_1 \cap \SL_2(\F_q) = \{ \begin{pmatrix} a \\ & a^{-1} \end{pmatrix} \begin{pmatrix} 1 & x \\ & 1 \end{pmatrix} \mid a \in A_{2^r m_0}, x \in \F_q \}.  
$$
(Note that this is up to conjugation only, but with a change of basis, we may assume it to be true.) 
Then
$$
[\SL_2(\F_q): H_1 \cap \SL_2(\F_q)] = [\SL_2(\F_q): H_{odd}] \cdot [H_{odd}: H_1 \cap \SL_2(\F_q)]
$$
$$
= m_0 \cdot p(\SL_2(\F_q)).
$$
By the proof of Lemma \ref{lemm32} we have
$$
[\GL_2(\F_q): H_1] = [\SL_2(\F_q): H_1 \cap \SL_2(\F_q)] \cdot [\GL_2(\F_q): H_1 \cdot \SL_2(\F_q)]
$$
$$
=m_0 \cdot p(\SL_2(\F_q)), 
$$
upon using  Lemma \ref{lemma35}. Again  as before if $m_0 >1$, we conclude $[\GL_2(\F_q):H_1] \geq 2$, and we get a contradiction. 
\end{proof}

Without loss of generality we may assume $H_1 \cap \SL_2(\F_q) = H_{odd}$.  

\

For $n$, $0 \leq n < q-1$, define a subgroup  $D(n) \subset \GL_2(\F_q)$ by 
$$
D(n) = \{ \begin{pmatrix} a^{n+1} \\ & a^{-n} \end{pmatrix} \mid a \in \F_q^\times \}. 
$$
Set 
$$
GH(n) = D(n) \cdot H_{odd}. 
$$
The subgroup $GH(0)$ is what we had called $GH_{odd}$ in Equation \eqref{GHODD-D1}.
\begin{lemm}\label{lemm5.6}
There is an $n$, $0 \leq n < q-1$, such that $H_1 = GH(n)$. 
\end{lemm}
\begin{proof} The proof of this lemma is in two steps. In the first step we show that $H_1$ is a subgroup of upper triangular matrices in $\GL_2(\F_q)$, and then we identify it explicitly. The simple argument we give for the first step was suggested independently by Roman Bezrukavnikov and Annette Pilkington. We start with the observation that $H_1 \cap \SL_2(\F_q)$ is normal in $H_1$. By Lemma \ref{lem5.5}, $H_1 \cap \SL_2(\F_q)$ consists of upper triangular matrices and contains all upper triangular unipotent matrices. Suppose $\begin{pmatrix} a & b \\ c & d \end{pmatrix} \in H_1$, and let $\begin{pmatrix} 1 & x \\ & 1 \end{pmatrix}$ be an arbitrary upper triangular unipotent matrix. Then since the matrix 
$$
\begin{pmatrix} a & b \\ c & d \end{pmatrix} \begin{pmatrix} 1 & x \\ & 1 \end{pmatrix} \begin{pmatrix} a & b \\ c & d \end{pmatrix}^{-1} = \begin{pmatrix} * & * \\ - \frac{c^2x}{ad-bc} & * \end{pmatrix}
$$ 
for all $x$, we must have $c =0$. 

Now we proceed to identify $H_1$ explicitly. By the proof of Lemma \ref{lemm32} and the statements of Lemmas \ref{lemma35} and \ref{lem5.5} we have 
\begin{equation}\label{eq3.2}
[\GL_2(\F_q): H_1] = [\SL_2(\F_q): H_{odd}]. 
\end{equation}
This means $|H_1| = (q-1)\cdot  |H_{odd}|$. So we need to find $(q-1)$ representatives for the quotient $H_1/H_{odd}$. By Lemma \ref{lemma35}, the determinant $\det: H_1 \to \F_q^\times$ is surjective. In particular, if $\varpi$ is a generator of $\F_q^\times$, there is a matrix $X$ in $H_1$, upper triangular by the first part, such that $\det X = \varpi$.  Since by Lemma \ref{lem5.5}, $H_1$ contains all upper triangular unipotent matrices, we may assume that $X$ is diagonal.  Since $\varpi$ is a generator of $\F_q^\times$, we may write $X =\begin{pmatrix} \varpi^{n+1} \\ & \varpi^m \end{pmatrix}$.  Since $\varpi = \det X = \varpi^{n+m+1}$, we conclude $n+m \equiv 0 \mod (q-1)$, or $m \equiv - n$.  So if we let $X_n = \begin{pmatrix} \varpi^{n+1} \\ & \varpi^{-n} \end{pmatrix}$, then $X_n \in H_1$ and $\det X_n = \varpi$. The elements $\{X_n^i \mid 0 \leq i < q-1\}$ provide the $(q-1)$ representatives for $H_1/H_{odd}$ that we need.  \end{proof}
\begin{coro}[From the proof]
We have 
$$
[\GL_2(\F_q):H_1] = p(\SL_2(\F_q)). 
$$
\end{coro}
\begin{proof} This is Equation \eqref{eq3.2}.
\end{proof}

\begin{lemm}\label{lemm5.8} We have 
$$
\core_{\GL_2(\F_q)}(H_1) = Z_{2^r}. 
$$
\end{lemm}
\begin{proof}
By Lemma \ref{lemm5.6}, it suffices to prove $\core_{\GL_2(\F_q)}(GH(n)) = Z_{2^r}$ for each $n$, and that means we need to determine $Z \cap GH(n)$.  Suppose we have an element of the form 
$$
t = \begin{pmatrix} a^{n+1} \\ & a^{-n} \end{pmatrix} \cdot \begin{pmatrix} b^{-1} \\ & b \end{pmatrix}, \quad a \in \F_q^\times, b \in A_{2^r}
$$
and suppose $t \in Z$. This means $a^{2n+1} = b^2$.  Write $a = \varpi^i, b = \varpi^{2^r \cdot j}$. Then we have 
$$
(2n+1) i \equiv 2^{r+1} j \mod (q-1). 
$$
Let $g = \gcd(2n+1, q-1)$. Then $j = gu$ for some $u$.  Then if $k$ is a multiplicative inverse of $(2n+1)/g$ modulo $(q-1)/g$, we have 
$$
i \equiv k \cdot 2^{r+1} \cdot u \mod \frac{q-1}{g}, 
$$
or 
$$
i =  k \cdot 2^{r+1} \cdot u +  \frac{q-1}{g} s 
$$
for some $s$.  So if for any $u, s$ we set $a = \varpi^i, b = \varpi^{2^r \cdot j}$ with $i, j$ satisfying 
$$
\begin{cases}
i = k \cdot 2^{r+1} \cdot u +  \frac{q-1}{g} s  \\
j = u g, 
\end{cases}
$$
then $a^{2n+1} = b^2$. Now we examine the matrix $t$. We see that $a^{-n} \cdot b$ is equal to $\varpi$ raised to the power 
$$
-n(k \cdot 2^{r+1} \cdot u +  \frac{q-1}{g} s ) + u \cdot g \cdot 2^{r}
$$
$$
= (-2 nk + g)2^r \cdot u - n \cdot \frac{q-1}{g} \cdot  s 
$$
$$
= 2^r \cdot \big\{  (-2 nk + g) \cdot u - n \cdot \frac{q-1}{2^r \cdot g} \cdot  s  \big\}. 
$$
We will show that 
\begin{equation}\label{eq32}
\gcd( 2 nk - g, n \cdot \frac{q-1}{g}  ) = 1. 
\end{equation}
Let us first look at $\gcd(2nk-g, n)$. This is equal to $\gcd (g, n)$ which is equal to $1$, as $g \mid 2n+1$ and $\gcd(n, 2n+1) = 1$. This means 
$$
\gcd( 2 nk - g, n \cdot \frac{q-1}{g}  ) = \gcd(2nk - g, \frac{q-1}{g})
$$
$$
= \gcd((2n+1)k - g - k, \frac{q-1}{g}) 
$$
$$
= \gcd(g \left\{\frac{2n+1}{g} \cdot k -1 \right\} -k, \frac{q-1}{g})
$$
$$
= \gcd(-k, \frac{q-1}{g})
$$
$$
=1. 
$$
In the above computation we have used the fact that $k$ is multiplicative inverse of $(2n+1)/g$ modulo $(q-1)/g$, so $k \cdot (2n+1)/g-1$ is divisible by $(q-1)/g$. Now that we have established Equation \eqref{eq32} we observe that since $-2nk + g$ is odd we have 
$$
 \gcd ( -2 nk + g , - n \cdot \frac{q-1}{2^r \cdot g} )=1. 
$$
This means that there are integers $s, u$ such that the corresponding $a, b$ satisfy $a^{-n} b = a^{n+1} b^{-1} = \varpi^{2^r}$, and that whenever $a^{-n} b = a^{n+1}b^{-1}$ for $a \in \F_q^\times, b \in A_{2^r}$, then the common value is of the form $\varpi^{ f \cdot 2^r}$ for some integer $f$. This finishes the proof of the lemma. 
\end{proof}

Now that we have identified the possibilities for $H_1$ and its core, we optimize the choices of $H_2, \dots, H_\ell$.  Define natural numbers $t_2, \dots, t_\ell$ by setting 
$$
Z \cap H_i = Z_{t_i}, \quad 2 \leq i \leq \ell. 
$$
We can pick each $t_i$ to be a divisor of $q-1$. By Equation \eqref{capst} and Lemma \ref{lemm38}, the statement 
$$
\core_{\GL_2(\F_q)} (H_1 \cap \dots \cap H_\ell) = \{e\}
$$
is equivalent to 
$$
\lcm(2^r, t_2, \dots, t_\ell) = q-1. 
$$
\begin{coro}\label{lemm39}
Suppose $t \mid q-1$, and $t \ne q-1$. Let $H(t)$ be the subgroup of $\GL_2(\F_q)$ with minimal $[\GL_2(\F_q): H]$ among the subgroups that satisfy $Z \cap H = Z_t$. Then
$$
H(t) = \begin{cases}
\GL_2(\F_q)^t & t \text{ odd}; \\
\GL_2(\F_q)^{2t} & t \text{ even}. 
\end{cases}
$$
Furthermore, 
$$
[\GL_2(\F_q): H(t)] = \begin{cases}
t & t \text{ odd}; \\
2t & t \text{ even}. 
\end{cases}
$$
\end{coro}

To finish the proof of Theorem \ref{theo11} we have to solve the following optimization problem for $n=q-1$.   
\begin{prob}\label{probgl2}
Suppose a natural number $n = 2^r m$ with $m$ odd is given. For a natural number $t$, set $\epsilon (t) = (3 + (-1)^t)/2$.  Find natural numbers $\ell, t_2, \dots, t_\ell$ such that 
\begin{itemize}
\item $\lcm(2^r, t_2, \dots, t_\ell) = n$;
\item $\sum_{i=2}^r \epsilon(t_i) t_i$ is minimal.  
\end{itemize}
\end{prob}
We call $(t_2, \dots, t_\ell)$ the {\em optimal choice} for $n$. 
\begin{lemm}
Write $n = 2^r p_1^{e_1} \cdots p_k^{e_k}$ with $p_i$'s distinct odd primes. Then  the optimal choice for $n$ is $(p_1^{e_1}, \dots, p_k^{e_k})$. 
\end{lemm}
\begin{proof}
Suppose $(t_2, \dots, t_\ell)$ is an optimal choice for $n$. If some $t_i$ is even, say equal to $2s$, replacing $t_i$ by $s$ does not change the $\lcm$ in the statement Problem \ref{probgl2}, but 
decreases the value of $\sum_i \epsilon(t_i) t_i$.  Since $(t_2, \dots, t_\ell)$ is optimal for $n$, this means that all of the $t_i$'s have to be odd. Next, write each $t_i$ as the product of prime powers $\pi_1^{m_1} \cdots \pi_v^{m_v}$. By Lemma \ref{lemm33}, $\sum_j \pi_j^{m_j} \leq t_i$ with equality only when $v=1$. Again, since $(t_2, \dots, t_\ell)$ is optimal, this means each $t_i$ is a prime power. It is also clear that if $i \ne j$, then $(t_i, t_j) =1$, because otherwise $t_i, t_j$ will be powers of the same prime, and we can throw away the one with smaller exponent. 
\end{proof}

Putting everything together we have proved the following theorem: 

\begin{theo}\label{theo313}
We have 
$$
p(\GL_2(\F_q)) = p(\SL_2(\F_q)) + T_{2,q}(q-1). 
$$
For $0 \leq n \leq q-2$, set 
$$
\mathcal C_n = \{GH(n), \GL_2(\F_q)^{p_1^{e_1}}, \dots, \GL_2(\F_q)^{p_k^{e_k}} \}.
$$
Up to conjugacy we have $q-1$ classes of minimal faithful collections for $\GL_2(\F_q)$ and they are given by $\mathcal C_n$, $0 \leq n \leq q-2$. \end{theo}

\section{Minimal faithful permutation representations of $\GL_3(\F_q)$}\label{sec:mingl3}
In this section we will be proving

\begin{theo}\label{theo314} 
We have 
$$
p(\GL_3(\F_q)) = p(\SL_3(\F_q)) + T_{3,q}(q-1). 
$$
For $0 \leq n \leq q-2$ set 
$$
\mathcal C_n = \{GH(n), \GL_3(\F_q)^{p_1^{e_1}}, \dots, \GL_3(\F_q)^{p_k^{e_k}} \}.
$$
Then up to conjugacy we have $q-1$ classes of minimal faithful sets for $\GL_3(\F_q)$ and they are given by the sets $\mathcal C_n$, $0 \leq n \leq q-2$.  \end{theo}

We use notation as per Lemma \ref{lemma35}, so we  let $\mathcal C = \{H_1, \dots, H_\ell\}$ be a minimal faithful collection of $\GL_3(\F_q)$ and let $H_1$ be such that  $H_1 \cap \slq$ doesn't contain the element of  $Z(\slq)$ of order $3$ if $3 \mid q-1$. Note, such $H_1$ is unique by Lemma \ref{lemma35}.

\begin{lemm}\label{lemm36}  
The subgroup
$
H_1 \cap \SL_3(\F_q)  
$
is a conjugate of $G_{3}$ in $\SL_{3}(\F_q)$, when $g = 3$ i.e. $3| q-1$.
\\ In the case of  $g = 1$ we have that 
$
H_1 \cap \SL_3(\F_q)  
$
is a conjugate of the subgroup $E_q^2:\GL_2(\F_q)$ of class $\mathscr{C}_1$.
\end{lemm}

\begin{proof}
$$
|H_1 \cdot \SL_{3}(\F_q)| = \frac{|H_1| \cdot |\SL_{3}(\F_q)|}{|H_1 \cap \SL_{3}(\F_q)|} = |\GL_{3}(\F_q)|.  
$$
We want the index of $H_1$ in $\GL_{3}(\F_q)$ i.e. $\frac{|\GL_{3}(\F_q)|}{|H_1|}$ to be minimized i.e. $
\frac{|\SL_{3}(\F_q)|}{|H_1 \cap \SL_{3}(\F_q)|} 
$ to be minimized. $H_1 \cap \SL_{{3}}(\F_q)$ is then the maximal core free subgroup of $\SL_3$, which we already showed is given by $G_{3}$.
Similarly, in the case of $g = 1$ we have that $H_1 \cap \SL_{{3}}(\F_q)$ is the maximal core free subgroup of $\SL_3$, given by $E_q^2:\GL_2(\F_q)$.
\end{proof}
Without loss of generality we may assume $H_1 \cap \SL_3(\F_q)= G_{3}$ for $g =3$ and  $H_1 \cap \SL_3(\F_q)= E_q^2:\GL_2(\F_q) $ for $g = 1$.  

\

For $n$, $0 \leq n < q-1$, define a subgroup  $D(n) \subset \glq$ by 
$$
D(n) = \{ \begin{pmatrix} a^{n+1} \\ & a^{1} \\ & & a^{-n} \end{pmatrix} \mid a \in \F_q^\times \}. 
$$
Set 
$$
GH(n) = D(n) \cdot G_{3},
$$ 
when $g = 3$ and 
$$
GH(n) = D(n) \cdot E_q^2:\GL_2(\F_q) 
$$ when $g = 1$.
\begin{lemm}\label{lemm37}
There is an $n$, $0 \leq n < q-1$, such that $H_1 = GH(n)$. 
\end{lemm}
\begin{proof} Observe that $H_1 \cap \SL_3(\F_q)$ is normal in $H_1$. By Lemma \ref{lemm36}, $H_1 \cap \SL_3(\F_q)$ consists of matrices with vanishing  entries in their last row's first column and second column and contains all 
upper triangular unipotent matrices. Suppose $\begin{pmatrix} a & b & c\\d & e & f\\ g & h & i  \end{pmatrix} \in H_1$, and let $\begin{pmatrix} 1 & 0 & x \\0 & 1 & 0\\ 0 & 0 & 1 \end{pmatrix}$ be an arbitrary upper triangular unipotent matrix. Then since the matrix 
$$
\begin{pmatrix} a & b & c\\d & e & f\\ g & h & i  \end{pmatrix}^{-1}\begin{pmatrix} 1 & 0 & x \\0 & 1 & 0\\ 0 & 0 & 1 \end{pmatrix} \begin{pmatrix} a & b & c\\d & e & f\\ g & h & i  \end{pmatrix} 
$$
$$
= \begin{pmatrix} * & *  & *\\ * & * & *\\ xg(dh-eg)/* & xh(dh-eg)/* & * \end{pmatrix}
$$
for all $x$, we must have $dh = eg$ if either of $g$ or $h$ are non-zero.
\\By symmetric considerations letting  $\begin{pmatrix} 1 & x & 0 \\0 & 1 & 0\\ 0 & 0 & 1 \end{pmatrix}$ be an arbitrary upper triangular unipotent matrix we obtain
$$
\begin{pmatrix} a & b & c\\d & e & f\\ g & h & i  \end{pmatrix}^{-1}\begin{pmatrix} 1 & 0 & 0 \\0 & 1 & x\\ 0 & 0 & 1 \end{pmatrix} \begin{pmatrix} a & b & c\\d & e & f\\ g & h & i  \end{pmatrix} 
$$
$$
= \begin{pmatrix} * & *  & *\\ * & * & *\\ xg(bg-ah)/*& xh(bg-ah)/* & * \end{pmatrix}
$$ 
for all $x$.
Hence, $bg = ah$ if either of $d$ or $e$ are non-zero.
In totality we have $dh = eg$ and $bg = ah$ if either of $g$ or $h$ are non-zero. However, in this scenario assuming $g \neq 0$ without loss of generality, $e = \frac{dh}{g}$ and $b = \frac{ah}{g}$, so $ea -bd = \frac{dha}{g} - \frac{dha}{g} = 0$. This implies that the determinant of the original matrix $\begin{pmatrix} a & b & c\\d & e & f\\ g & h & i  \end{pmatrix}$ vanishes.  Contradiction! Hence, both $g$ and $h$ equal $0$.
Now we proceed to identify $H_1$ explicitly. By the proof of Lemma \ref{lemm32} and the statements of Lemmas \ref{lemma35} and \ref{lemm36} we have 
\begin{equation}\label{eq32-2}
[\GL_3(\F_q): H_1] = [\SL_3(\F_q): G_{3}]. 
\end{equation} 
This means $|H_1| = (q-1)\cdot  |G_{3}|$. So we need to find $(q-1)$ representatives for the quotient $H_1/G_{3}$. By Lemma \ref{lemma35}, the determinant $\det: H_1 \to \F_q^\times$ is surjective. In particular, if $\varpi$ is a generator 
of $\F_q^\times$, there is a matrix $X$ in $H_1$, with vanishing  entries in its first column's second row and third row by the first part, 
such that $\det X = \varpi$.  Since by Lemma \ref{lemm36}, $H_1$ contains all upper triangular uni potent matrices, we may assume that $X$ is diagonal.  Since $\varpi$ is a generator of $\F_q^\times$, we may write $X =\begin{pmatrix} \varpi^{n+1} \\ & \varpi^m \\ && \varpi^k \end{pmatrix}$.  
Since $\varpi = \det X = \varpi^{n+m+k+1}$, we conclude $n+m+k \equiv 0 \mod (q-1)$, or $m+k \equiv - n$.  So if we let $X_n = 
\begin{pmatrix} \varpi^{n+1} \\ & \varpi^{-n -1}\\ & & \varpi^{1} \end{pmatrix}$, then $X_n \in H_1$ and $\det X_n = \varpi$. The elements 
$\{X_n^i \mid 0 \leq i < q-1\}$ provide the $(q-1)$ representatives for $H_1/G_{3}$ 
that we need.
\\ Note, when $g=1$ replacing $G_{3}$ by  $E_q^2:\GL_2(\F_q)$ and applying the same reasoning as above gives us the desired conclusion. \end{proof}
\begin{coro}[From the proof] 
{We have 
$$
[\GL_{3}(\F_q):H_1] = (q-1)\cdot p(\SL_{3}(\F_q)). 
$$}

\end{coro}

\begin{lemm}\label{lemm38}  
{$$
\core_{\GL_{3}(\F_q)}(H_1) = Z_{3^r}. 
$$} where $r$ is the highest power of $3$ dividing $q -1$.
\end{lemm}
\begin{proof}
By Lemma \ref{lemm37}, it suffices to prove $\core_{\glq}(GH(n)) = Z_{3^r}$ for each $n$, and that means we need to determine $Z \cap GH(n)$.  Suppose we have an element of the form 
$$
t = \begin{pmatrix} a^{n+1} \\ & a^{-n -1} \\ & & a^{1} \end{pmatrix} \cdot \begin{pmatrix} b^{-2} \\ & b \\ & & b \end{pmatrix}, \quad a \in \F_q^\times, b \in A_{3^r}
$$
and suppose $t \in Z$. This means $a^{2n+2} = b^3$ and $a^{n} = b^3 $. 
 Write $a = \varpi^i, b = \varpi^{3^r \cdot j}$. Then we have 
$$
 (2n+2)\cdot i \equiv n\cdot i  \equiv3^{r+1} j \mod (q-1). 
$$
Let $g = \gcd(2n+2, q-1)$. Then $j = gu$ for some $u$.  Then if $k$ is a multiplicative inverse of $(2n+2)/g$ modulo $(q-1)/g$, we have 
$$
i \equiv k \cdot 3^{r+1} \cdot u \mod \frac{q-1}{g}, 
$$
or 
$$
i =  k \cdot 3^{r+1} \cdot u +  \frac{q-1}{g} s 
$$
for some $s$.  So if for any $u, s$ we set $a = \varpi^i, b = \varpi^{3^r \cdot j}$ with $i, j$ satisfying 
$$
\begin{cases}
i = k \cdot 3^{r+1} \cdot u +  \frac{q-1}{g} s  \\
j = u g, 
\end{cases}
$$
and $q-1| (n+2)\cdot i$
then $a^{2n+2} = a^{n} = b^{3}$. Now we examine the matrix $t$. We see that $a^{-n-1}\cdot b$ is equal to $\varpi$ raised to the power 
$$
(-n-1)(k \cdot 3^{r+1} \cdot u +  \frac{q-1}{g} s ) + u \cdot g \cdot 3^{r}
$$
$$
= ({3\cdot(-n-1)\cdot k \cdot u} + u\cdot g)3^r+ \frac{q-1}{g} \cdot  s 
$$
$$
= 3^r \cdot \big\{  (3 \ (-n-1)\cdot {k} + g) \cdot u + \frac{q-1}{3^r \cdot g} \cdot  s  \big\}. 
$$
We will show that 
\begin{equation}\label{eq32-4}
\gcd( 3\cdot(-n-1) \cdot k + g, \frac{q-1}{g}  ) = 1. 
\end{equation}
We have 
$$
 \gcd(3 \cdot(-n-1)\cdot k + g, \frac{q-1}{g})
$$
$$
= \gcd(2\cdot (-n-1)\cdot k + g +(-n-1)\cdot k, \frac{q-1}{g}) 
$$
$$
= \gcd(g \left\{\frac{2\cdot (-n-1) \cdot k}{g}+1 \right\} + (-n-1) \cdot k, \frac{q-1}{g})
$$
$$
= \gcd((-n-1)\cdot k, \frac{q-1}{g})
$$
$$
=1. 
$$
In the above computation we have used the fact that $k$ is multiplicative inverse of $(2n+2)/g$ modulo $(q-1)/g$, so $2\cdot (-n-1) \cdot k/g+ 1$ is divisible by $(q-1)/g$. Now that we have established Equation \eqref{eq32} we observe that
$$
 \gcd ( 3\cdot (-n-1) \cdot k + g , \cdot \frac{q-1}{3^r \cdot g} )=1. 
$$
This means that there are integers $s, u$ such that the corresponding $a, b$ satisfy $a^{n+1} b^{-2} = a^{-n-1} b^{1}= ab = \varpi^{3^r}$, and that whenever $a^{n+1} b^{-2} = a^{-n-1} b^{1}= ab$ for $a \in \F_q^\times, b \in A_{3^r}$, then the common value is of the form $\varpi^{ f \cdot 3^r}$ for some integer $f$. This finishes the proof of the lemma. 
\end{proof}

{Our goal is to minimize 
$$
[\GL_{3}(\F_q): H_2] + \dots + [\GL_{3}(\F_q): H_\ell]. 
$$}

We use the following corollary to Lemma \ref{lemma39} 
\begin{coro}
Suppose $t|q-1$ , and $t \ne q-1$ , then in the case of $n = 3$, we have $g = 1$ or $3$. So, the subgroup $H(t)$ of $\GL_3(\F_q)$ with minimal $[\GL_3(\F_q): H]$ among the subgroups that satisfy $Z \cap H = Z_t$ is given by
$$
H(t) = \begin{cases}
\GL_3(\F_q)^t & 3 \nmid t; \\
\GL_3(\F_q)^{3t} & 3|t. 
\end{cases}
$$
Furthermore, 
$$
[\GL_3(\F_q): H(t)] = \begin{cases}
t & 3 \nmid t; \\
3t& 3|t. 
\end{cases}
$$
\end{coro}

To finish the proof of Theorem \ref{theo313} we have to solve the following optimization problem for $n=q-1$.   
\begin{prob}\label{prob}  
Suppose a natural number $n = 3^r m$ with $m$ coprime to $3$ is given. For a natural number $t$, set $\epsilon (t) = (4 + 2(-1)^t)/2$.  Find natural numbers $\ell, t_2, \dots, t_\ell$ such that 
\begin{itemize}
\item $\lcm(3^r, t_2, \dots, t_\ell) = n$;
\item $\sum_{i=2}^r \epsilon(t_i) t_i$ is minimal.  
\end{itemize}
\end{prob}
We call $(t_2, \dots, t_\ell)$ the {\em optimal choice} for $n$. 
\begin{lemm}
Write $n = 3^r p_1^{e_1} \cdots p_k^{e_k}$ with $p_i$'s distinct odd primes. Then  the optimal choice for $n$ is $(p_1^{e_1}, \dots, p_k^{e_k})$. 
\end{lemm}
\begin{proof}
Suppose $(t_2, \dots, t_\ell)$ is an optimal choice for $n$. If some $t_i$ is even, say equal to $3s$, replacing $t_i$ by $s$ does not change the $\lcm$ in the statement Problem \ref{prob}, but 
decreases the value of $\sum_i \epsilon(t_i) t_i$.  Since $(t_2, \dots, t_\ell)$ is optimal for $n$, this means that all of the $t_i$'s have to be coprime to $3$. Next, write each $t_i$ as the product of prime powers $\pi_1^{m_1} \cdots \pi_v^{m_v}$. By Lemma \ref{lemm33}, $\sum_j \pi_j^{m_j} \leq t_i$ with equality only when $v=1$. Again, since $(t_2, \dots, t_\ell)$ is optimal, this means each $t_i$ is a prime power. It is also clear that if $i \ne j$, then $(t_i, t_j) =1$, because otherwise $t_i, t_j$ will be powers of the same prime, and we can throw away the one with smaller exponent. 
\end{proof}

Putting everything together finishes the proof of Theorem \ref{theo314}.

\section{General remarks/Future Work}
Although we do not have a way to tackle the conjecture for general $n$ at least for a \textit{very divisible} prime $n$ the algorithm described in \S\ref{sec:mingl2} and \S\ref{sec:mingl3} extends to other $\GL_n(\F_q)$ for small $n$. We observe that for $n$ prime we can always isolate a subgroup, in our minimal faithful collection, whose order is coprime to $g = \gcd(n,q-1)$. Call this subgroup $H_1$. Then, we may analyze $H_1 \cap \SL_n(\F_q)$ and observe that it's the largest core free subgroup of $\SL_n$. To compute it explicitly we use Patton's result in \cite{Patton}. Namely, that the maximal subgroup $P$ of $\SL_n(\F_q)$ has the form 
$$\begin{pmatrix}
  B & x\\ 
  0 & b
\end{pmatrix}$$
where, $b^{-1} =  \det(B) \in \F_q^{\times}$, $x \in \F_q^{n-1}, B \in \GL_{n-1}(\F_q)
$ and $b,x$ arbitrary. Then, at least heuristically, the maximal core free subgroup $H_1$ must be a subgroup of $P$.
This forces $A$ to be the maximal  subgroup of $\GL_{n-1}(\F_q)$, with order coprime to $g$. We can express $A$ as $D_1\cdot A'$, where $A'$ is the maximal subgroup of $\SL_{n-1}(\F_q)$ with order coprime to $g$. Now we apply Patton's result again and iterate the procedure. We keep doing this until we obtain a $2 \times 2$ matrix, which we can handle by computing the maximal subgroup of $\SL_2(\F_q)$ of order coprime to $g$. The case of general $n$ will require  new techniques other than those described in this paper.

\end{document}